\documentclass{article}

\usepackage{arxiv}

\usepackage[utf8]{inputenc} 
\usepackage[T1]{fontenc}    
\usepackage{hyperref}       
\usepackage{url}            
\usepackage{booktabs}       
\usepackage{amsmath}
\usepackage{mathrsfs}
\usepackage{amsfonts}       
\usepackage{nicefrac}       
\usepackage{microtype}      
\usepackage{graphicx}
\usepackage{natbib}
\usepackage{doi}
\usepackage{bm}
\usepackage{bbm}
\usepackage{enumitem}

\title{Entropic Strict Minimum Message Length and Its Connections to PAC--Bayes and NML
}

\author{ \href{https://orcid.org/0000-0003-3017-0871}{\includegraphics[scale=0.06]{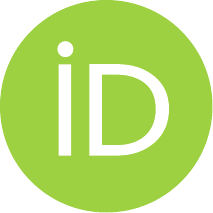}\hspace{1mm}Enes Makalic}
    \\
    Faculty of Information Technology\\
	Monash University\\
	Clayton, VIC 3800 \\
	\texttt{enes.makalic@monash.edu} \\
	\And
	\href{https://orcid.org/0000-0002-1788-2375}{\includegraphics[scale=0.06]{orcid.pdf}\hspace{1mm}Daniel F.~Schmidt} \\
    Faculty of Information Technology\\
	Monash University\\
	Clayton, VIC 3800 \\
	\texttt{daniel.schmidt@monash.edu} \\
}

\DeclareMathOperator*{\argmin}{arg\,min}
\DeclareMathOperator*{\argmax}{arg\,max}

\newtheorem{lemma}{Lemma}
\newtheorem{corollary}{Corollary}
\newtheorem{proposition}{Proposition}
\newtheorem{assumption}{Assumption}
\newtheorem{theorem}{Theorem}
\newtheorem{remark}{Remark}

\newenvironment{proof}[1][Proof]{\par\noindent\textbf{#1.} }{\hfill$\square$\par}

\renewcommand{\shorttitle}{Entropic SMML and Its Connections to PAC--Bayes and NML}

\hypersetup{
pdftitle={The Wallace--Freeman invariant point estimator},
pdfsubject={stat.ME, stat.TH},
pdfauthor={Enes Makalic, Daniel F.~Schmidt},
pdfkeywords={Maximum likelihood, Minimum Message Length, Bias, 
Consistency, Weibull distribution},
}

\begin{document}
\maketitle

\begin{abstract}
We introduce entropic strict minimum message length (SMML), a risk-sensitive generalization of strict minimum message length coding. The proposed criterion replaces expected two-part codelength under the prior predictive distribution with an exponential certainty equivalent, thereby defining a one-parameter family of coding rules that interpolates between Bayesian average-case coding and worst-case minimax coding. We show that ordinary SMML is recovered in the risk-neutral limit, while the extreme risk-sensitive limit yields a minimax codelength criterion.
Applying the same entropic soft maximum to regret
relative to the oracle maximum likelihood codelength recovers the normalized maximum likelihood (NML) minimax-regret principle.
We further prove that entropic SMML admits a variational characterization as a Kullback--Leibler-regularized worst-case expected codelength, giving it a PAC--Bayes-type interpretation. 
We establish joint \(n\)--\(\tau\) asymptotics that identify how the risk parameter must scale with sample size in order to recover Bayesian average-case, intermediate robust, and worst-case minimax coding behavior.
For regular exponential families, the fixed-codebook partition remains affine in sufficient-statistic space, while the codepoints satisfy a tilted moment-matching condition and admit an interpretation as tilted Bregman centroids. These results position entropic SMML as an information theoretic bridge between MML, PAC--Bayes, and MDL.
\end{abstract}

\section{Introduction}
\label{sec:introduction}
Minimum message length (MML)~\cite{WallaceBoulton68,WallaceBoulton75,WallaceFreeman87,Wallace05} and minimum description length (MDL)~\cite{Rissanen78,Rissanen84,Rissanen96,Rissanen01,Rissanen07,Grunwald07,GrunwaldRoos19} are two of the most influential coding-based approaches to statistical inference and model selection. Both arise from the principle that learning may be viewed as compression, yet they differ in how uncertainty and optimality are formalized. In its strict two-part form, strict minimum message length (SMML)~\cite{WallaceBoulton75,Wallace05,MakalicSchmidt26b} builds a finite codebook by minimizing the \emph{expected} codelength under the prior predictive distribution, and is therefore inherently Bayesian and average-case in nature. By contrast, normalized maximum likelihood (NML)~\cite{Rissanen96,Rissanen07,GrunwaldRoos19}, the canonical code in MDL, is characterized by a minimax regret principle and optimizes worst-case performance relative to the maximum likelihood (ML) code. Although these paradigms are often presented as reflecting different statistical philosophies, both are coding rules for assigning codelength to data.

This paper introduces a coding principle that bridges these regimes by replacing the expected codelength in strict MML with an \emph{entropic}, or risk-sensitive, certainty equivalent. The resulting criterion is indexed by a risk parameter \(\tau>0\): small values of \(\tau\) recover average-case Bayesian behavior, while larger values increasingly penalize rare but costly codewords.
This entropic deformation is natural from several perspectives. First, exponential certainty equivalents are canonical in risk-sensitive decision theory, statistical physics, large deviations, and the theory of entropic risk measures, where they provide a smooth interpolation between expectation and maximization~\cite{DonskerVaradhan75,FollmerKnispel11}. Second, the new criterion admits a variational representation in terms of Kullback--Leibler (KL) divergence~\cite{KullbackLeibler51}, namely the Gibbs--Donsker--Varadhan variational formula~\cite{DonskerVaradhan75,Catoni07}, and therefore has a PAC--Bayes-type interpretation as a KL-regularized worst-case expected codelength over prior predictive tilts. Third, in the extreme risk-sensitive limit, the criterion reduces to a minimax codelength objective. After centering codelengths by the oracle maximum likelihood code, the same entropic construction, when applied to regret rather than absolute codelength and optimized over all coding distributions, recovers the NML minimax regret principle.

The geometric motivation for this construction comes from recent asymptotic results for ordinary SMML in regular parametric models~\cite{MakalicSchmidt26b}. The present paper shows that information-geometric structure of ordinary SMML is preserved under the entropic deformation, even though the coding criterion is altered. 

The main contributions of the paper are fourfold. We first derive a KL-variational representation of the entropic objective, giving it a PAC--Bayesian interpretation as a KL-regularized robust codelength criterion. We then characterize its endpoint limits, distinguishing the large-\(\tau\) worst-case absolute-codelength limit from the regret-centered limit that recovers NML. Next, we establish joint \(n\)--\(\tau\) asymptotics describing the transition between Bayesian, robust, and minimax regimes. Finally, for regular exponential families, we show that fixed-codebook partitions remain affine in sufficient-statistic space and that the codepoints are tilted Bregman centroids~\cite{BanerjeeEtAl05}.

The remainder of the paper is organized as follows. Section~\ref{sec:smml} reviews ordinary SMML and introduces the entropic criterion. Section~\ref{sec:var:pac:bayes} gives the variational and PAC--Bayesian characterization of entropic SMML. Section~\ref{sec:smml:nml} establishes the endpoint limits connecting entropic SMML to ordinary SMML and to NML. Section~\ref{sec:n:tau:asymptotics} states the joint $n$--$\tau$ asymptotic theorem. Section~\ref{sec:exponential:families} specializes the theory to regular exponential families,  derives the tilted Bregman centroid interpretation, and applies the theory to the binomial distribution as an example. Section~\ref{sec:conclusion} concludes with a discussion of implications for coding-based statistical inference.
\section{Strict Minimum Message Length}
\label{sec:smml}
This section fixes notation and defines the two coding criteria studied in the paper. We work throughout with a countable data space $\mathcal X_n$, as in the strict minimum message length (SMML) framework~\cite{WallaceBoulton75,Wallace05,MakalicSchmidt26b}. For the ordinary SMML criterion, the basic objects are a finite partition of $\mathcal X_n$, a collection of codepoints, and the resulting two-part codelength under the prior predictive distribution. The ordinary SMML setup and its KL-projection interpretation are reviewed in~\cite{MakalicSchmidt26b}. 

Let $X_{n}$ denote a random dataset of size $n$, taking values in the countable space $\mathcal X_n$, and let $x\in\mathcal X_n$ denote a realization. Let $\{p_n(x\mid\theta):\theta\in\Theta\}$ be a parametric model with prior $\pi(\theta)$, and define the prior predictive distribution  
\begin{align}
\label{eqn:prior-predictive}
    r_n(x) = \int_\Theta p_n(x\mid\theta) \pi(\theta) d\theta.
\end{align}
We assume throughout that $r_n(x)>0$ for all $x\in\mathcal X_n$. Let $\mathcal C_n$ denote the class of admissible finite codebooks at sample size $n$. An admissible finite codebook is a triple $(\mathcal P,q,\theta) \in \mathcal{C}_n$, where
\begin{itemize}
\item $\mathcal P=\{P_1,\dots,P_k\}$ is a finite partition of $\mathcal X_n$ into nonempty cells, 
\item $q=(q_1,\dots,q_k)$ satisfies $q_j>0$ and $\sum_{j=1}^k q_j=1$, 
\item $\theta=(\theta_1,\dots,\theta_k)$ contains one codepoint $\theta_j\in\Theta$ per cell.
\end{itemize}
We write $j(x)$ for the unique index such that $x\in P_{j(x)}$.
Given $(\mathcal P,q,\theta)\in\mathcal C_n$, the associated two-part SMML codelength is
\begin{align}
\label{eqn:two-part-code}
\Lambda_{\mathcal P,q,\theta}(x)
=
-\log q_{j(x)}-\log p_n(x\mid \theta_{j(x)}).
\end{align}
The first term encodes the cell index (the \emph{assertion}), and the second encodes the data conditional on the selected codepoint (the \emph{detail}). The ordinary SMML objective is the expected codelength under the prior predictive law:
\begin{align}
\label{eqn:ordinary-smml-objective}
\mathcal I_n(\mathcal P,q,\theta)
=
\mathbb E_{r_n}
\!\left[
\Lambda_{\mathcal P,q,\theta}(X_{n})
\right].
\end{align}
An SMML codebook is any minimizer
\begin{align}
\label{eqn:ordinary-smml-argmin}
(\mathcal P_n^{*},q_n^{*},\theta_n^{*})
\in
\argmin_{(\mathcal P,q,\theta)\in\mathcal C_n}
\mathcal I_n(\mathcal P,q,\theta).
\end{align}

The codebook formulation in \eqref{eqn:ordinary-smml-argmin} is equivalent to the classical partition-based formulation of SMML. Indeed, for fixed $(\mathcal P,\theta)$, the minimizing assertion probabilities are
\begin{align}
\label{eqn:ordinary-q-opt}
q_j^*
=
r_n(P_j)
:=
\sum_{x\in P_j} r_n(x),
\qquad j=1,\dots,k,
\end{align}
and for fixed $(\mathcal P,q)$, each cellwise codepoint solves
\begin{align}
\label{eqn:ordinary-codepoint}
\theta_j^*
\in
\argmax_{\theta\in\Theta}
\sum_{x\in P_j} r_n(x)\log p_n(x\mid\theta),
\quad j=1,\dots,k.
\end{align}
Equivalently, each ordinary SMML codepoint is the KL projection of the normalized cellwise distribution onto the model family~\cite{MakalicSchmidt26b}.
\subsection{Entropic Strict Minimum Message Length}
We now introduce the entropic, or risk-sensitive, generalization of SMML. For $\tau>0$, define the entropic SMML objective
\begin{align}
\label{eqn:entropic-smml-objective}
\mathcal I_{n,\tau}(\mathcal P,q,\theta)
=
\frac{1}{\tau}
\log
\mathbb E_{r_n}
\!\left[
\exp\!\big(
\tau\,\Lambda_{\mathcal P,q,\theta}(X_{n})
\big)
\right].
\end{align}
An \emph{entropic SMML} codebook is any minimizer
\begin{align}
\label{eqn:entropic-smml-argmin}
(\mathcal P_{n,\tau}^*,q_{n,\tau}^*,\theta_{n,\tau}^*)
\in
\argmin_{(\mathcal P,q,\theta)\in\mathcal C_n}
\mathcal I_{n,\tau}(\mathcal P,q,\theta).
\end{align}
The parameter $\tau$ controls sensitivity to large codelengths. For small $\tau$, the criterion in \eqref{eqn:entropic-smml-objective} behaves like the ordinary SMML expected codelength. For large $\tau$, it increasingly emphasizes the upper tail of the codelength distribution and approaches a worst-case coding principle. 

For a fixed codebook, the partition rule remains pointwise minimum codelength:
\begin{align}
\label{eqn:fixed-codebook-partition}
x\in P_j
\iff
j\in
\argmin_{1\le \ell\le k}
\left\{
-\log q_\ell-\log p_n(x\mid\theta_\ell)
\right\},
\end{align}
so the entropic deformation does not alter the local decision rule, only the global criterion used to choose the codebook.

For a fixed partition $\mathcal P=\{P_1,\dots,P_k\}$, define
\begin{align}
\label{eqn:Ajtau}
C_{j,\tau}(\theta)
:=
\sum_{x\in P_j}
r_n(x)\,p_n(x\mid\theta)^{-\tau},
\qquad j=1,\dots,k.
\end{align}
Then \eqref{eqn:entropic-smml-objective} may be rewritten as
\begin{align}
\label{eqn:entropic-fixed-partition}
\mathcal I_{n,\tau}(\mathcal P,q,\theta)
=
\frac{1}{\tau}
\log
\sum_{j=1}^k
q_j^{-\tau} C_{j,\tau}(\theta_j),
\end{align}
and the fixed-partition codepoint optimization becomes
\begin{align}
\label{eqn:entropic-codepoint}
\theta_{j,\tau}^*
\in
\argmin_{\theta\in\Theta}
C_{j,\tau}(\theta)
=
\argmin_{\theta\in\Theta}
\sum_{x\in P_j}
r_n(x)\,p_n(x\mid\theta)^{-\tau},
\end{align}
for $j=1,\dots,k$. This is the risk-sensitive analogue of the ordinary SMML cellwise fit in \eqref{eqn:ordinary-codepoint}. The corresponding fixed-partition assertion probabilities are derived in the next section.
\subsection{Connection to generalized source coding}
The entropic SMML criterion defined in \eqref{eqn:entropic-smml-objective} is closely related to Campbell's exponential
source coding criterion~\cite{Campbell65} where the ordinary expected code length is replaced by an exponential average of the form
\[
  L_t(\ell)
  =
  \frac{1}{t}
  \log
  \sum_i p_i e^{t\ell_i},
  \qquad t>0,
\]
where \(p_i\) is the source probability of symbol \(i\) and \(\ell_i\) is the corresponding codeword length. Campbell showed that the optimal value of this criterion is governed by R\'{e}nyi entropy~\cite{Renyi61} with order \(\alpha=(1+t)^{-1}\), providing an operational source coding interpretation of R\'{e}nyi entropy that is analogous to Shannon's theorem for ordinary expected length.

The entropic SMML criterion may be viewed as the same exponential-type deformation applied not to arbitrary codeword lengths, but to SMML two-part codelengths. Indeed, for a fixed admissible codebook \((\mathcal P,q,\theta)\), define 
\[
  Q_{\mathcal P,q,\theta}(x)
  =
  q_{j(x)}p_n(x\mid\theta_{j(x)}),
\]
so that
\[
  \Lambda_{\mathcal P,q,\theta}(x)
  =
  -\log Q_{\mathcal P,q,\theta}(x).
\]
Then the entropic SMML objective can be written as
\[
  \mathcal I_{n,\tau}(\mathcal P,q,\theta)
  =
  \frac{1}{\tau}
  \log
  \sum_{x\in\mathcal X_n}
  r_n(x)
  Q_{\mathcal P,q,\theta}(x)^{-\tau},
\]
which is Campbell's exponential-average length criterion with the source distribution \(p_i\) replaced by the prior predictive distribution \(r_n(x)\), and with ordinary codeword lengths replaced by SMML two-part codelengths.  If optimization were carried out over all probability distributions \(Q\) on \(\mathcal X_n\), the optimal value would be
\[
  \inf_Q
  \frac{1}{\tau}
  \log
  \sum_x r_n(x)Q(x)^{-\tau}
  =
  H_{1/(1+\tau)}(r_n),
\]
where
\[
  H_\alpha(r_n)
  =
  \frac{1}{1-\alpha}
  \log
  \sum_x r_n(x)^\alpha,
  \qquad
  \alpha=\frac{1}{1+\tau},
\]
is the R\'{e}nyi entropy~\cite{Renyi61}. Since SMML restricts \(Q\) to the two-part statistical form \(Q_{\mathcal P,q,\theta}(x)=q_{j(x)}p_n(x\mid\theta_{j(x)})\),  we have
\[
  \mathcal I_{n,\tau}(\mathcal P,q,\theta)
  \ge
  H_{1/(1+\tau)}(r_n)
\]
for every admissible codebook. Thus ordinary SMML is lower bounded by Shannon entropy, while entropic SMML is lower bounded by the corresponding R\'{e}nyi entropy for finite \(\tau>0\).

This viewpoint clarifies the role of the risk parameter. Campbell's exponential length criterion~\cite{Campbell65} penalizes long codewords more strongly than ordinary expected length; analogously, entropic SMML penalizes large two-part statistical codelengths. As \(\tau\downarrow0\), the criterion reduces to the ordinary prior-predictive expected codelength. As \(\tau\to\infty\), it approaches a worst-case codelength criterion. Thus entropic SMML can be interpreted as lifting the exponential source coding principle from ordinary source codes to structured statistical two-part codes. 

\begin{remark}[Escort interpretation]
Campbell's exponential source coding theorem~\cite{Campbell65} admits a natural interpretation in terms of escort distributions~\cite{Bercher09,Csiszar95}. In the unconstrained coding problem, minimizing the exponential-average codelength
\[
  \frac{1}{\tau}
  \log
  \sum_x r_n(x)Q(x)^{-\tau}
\]
over all probability distributions \(Q\) yields the optimal solution
\[
  Q_\tau^*(x)
  =
  \frac{r_n(x)^{1/(1+\tau)}}
       {\sum_y r_n(y)^{1/(1+\tau)}},
\]
which is the escort distribution of the prior predictive law \(r_n\) of order \(\alpha=(1+\tau)^{-1}\). Entropic SMML imposes the additional structural constraint that the coding distribution factorizes as
\[
  Q_{\mathcal P,q,\theta}(x)
  =
  q_{j(x)}p_n(x\mid\theta_{j(x)}),
\]
so that the entropic SMML problem may be viewed as a constrained escort-coding problem. The corresponding redundancy
\[
  \Delta_{n,\tau}^{\mathrm{SMML}}
  =
  \mathcal I_{n,\tau}^{\mathrm{SMML}}
  -
  H_{1/(1+\tau)}(r_n)
\]
can be interpreted as the R\'{e}nyi divergence projection error of the escort distribution onto the SMML two-part code class. Within each cell, the tilted moment-matching weights
\[
  w_{j,\tau}(x;\theta)
  \propto
  r_n(x)p_n(x\mid\theta)^{-\tau}
\]
induce a local, model-dependent escort tilt of the prior predictive
distribution. 
\end{remark}

The next section gives a complementary variational interpretation of entropic SMML. The same exponential criterion is equivalent to a Kullback--Leibler (KL)-regularized~\cite{KullbackLeibler51} worst-case expected codelength over prior predictive tilts.
\section{Variational and PAC--Bayesian Characterization}
\label{sec:var:pac:bayes}
We now give the variational representation of entropic SMML and derive its PAC--Bayesian interpretation. The key observation is that the exponential certainty equivalent in~\eqref{eqn:entropic-smml-objective} admits the Gibbs--Donsker--Varadhan KL dual form, turning entropic SMML into a KL-regularized robust codelength criterion~\cite{DonskerVaradhan75,Catoni07}. We then use this representation to derive the optimal assertion probabilities for a fixed partition.

Let $\mathcal M(\mathcal X_n)$ denote the set of probability distributions on the countable data space $\mathcal X_n$, and for $s,r\in\mathcal M(\mathcal X_n)$ write $s\ll r$ if $s$ is absolutely continuous with respect to $r$. For such $s$ and $r$, define the KL divergence~\cite{KullbackLeibler51}
\begin{align}
\label{eqn:kl-def}
D_{\mathrm{KL}}(s\|r)
=
\sum_{x\in\mathcal X_n}
s(x)\log \frac{s(x)}{r(x)},
\end{align}
with the usual convention that $0\log 0 = 0$. The following theorem is the basic variational identity underlying the entropic criterion.

\begin{theorem}[Variational representation of entropic SMML]
\label{thm:variational}
For any admissible codebook $(\mathcal P,q,\theta)\in\mathcal C_n$ and any $\tau>0$ such that
\[
Z_{n,\tau}:=\mathbb{E}_{r_n}\exp\{\tau\Lambda_{\mathcal P,q,\theta}(X_n)\}<\infty,
\]
$\mathcal I_{n,\tau}(\mathcal P,q,\theta)$ can be written as
\begin{align}
\label{eqn:variational-smml}
\sup_{s\in\mathcal M(\mathcal X_n):\, s\ll r_n}
\left\{
\mathbb E_s\!\left[\Lambda_{\mathcal P,q,\theta}(X_{n})\right]
-
\frac{1}{\tau}D_{\mathrm{KL}}(s\|r_n)
\right\}.
\end{align}
Moreover, the supremum is attained by the exponentially tilted distribution
\begin{align}
\label{eqn:tilted-measure}
s_{n,\tau}^*(x)
=
\frac{
r_n(x)\exp\!\big(\tau \Lambda_{\mathcal P,q,\theta}(x)\big)
}{
\sum_{y\in\mathcal X_n}
r_n(y)\exp\!\big(\tau \Lambda_{\mathcal P,q,\theta}(y)\big)
},
\quad x\in\mathcal X_n.
\end{align}
\end{theorem}
\begin{proof}
Fix $(\mathcal P,q,\theta)\in\mathcal C_n$ and write $\Lambda(x)=\Lambda_{\mathcal P,q,\theta}(x)$. For any $s\in\mathcal M(\mathcal X_n)$ with $s\ll r_n$, $\mathbb E_s[\Lambda(X_{n})] - \frac{1}{\tau}D_{\mathrm{KL}}(s\|r_n)$ can be written as
\begin{align}
\sum_{x\in\mathcal X_n}
s(x)\Lambda(x)
-
\frac{1}{\tau}
\sum_{x\in\mathcal X_n}
s(x)\log \frac{s(x)}{r_n(x)}
=
\frac{1}{\tau}
\sum_{x\in\mathcal X_n}
s(x)
\log
\frac{r_n(x)e^{\tau \Lambda(x)}}{s(x)}.
\label{eqn:variational-proof-1}
\end{align}
Define
\begin{align}
\label{eqn:Zn}
Z_{n,\tau}
:=
\sum_{y\in\mathcal X_n}
r_n(y)e^{\tau \Lambda(y)}.
\end{align}
Then
\begin{align}
\frac{r_n(x) e^{\tau\Lambda(x)}}{s(x)}
=
Z_{n,\tau}\,
\frac{s_{n,\tau}^*(x)}{s(x)},
\label{eqn:variational-proof-2}
\end{align}
where $s_{n,\tau}^*$ is given by \eqref{eqn:tilted-measure}. Substituting \eqref{eqn:variational-proof-2} into \eqref{eqn:variational-proof-1} yields 
\begin{align}
\mathbb E_s[\Lambda]
-
\frac{1}{\tau}D_{\mathrm{KL}}(s\|r_n)
=
\frac{1}{\tau}\log Z_{n,\tau}
-
\frac{1}{\tau}D_{\mathrm{KL}}(s\|s_{n,\tau}^*).
\label{eqn:variational-proof-3}
\end{align}
Since KL divergence is nonnegative, the right-hand side is maximized when $s=s_{n,\tau}^*$, proving
\[
\sup_{s\ll r_n}
\left\{
\mathbb E_s[\Lambda]
-
\frac{1}{\tau}D_{\mathrm{KL}}(s\|r_n)
\right\}
=
\frac{1}{\tau}\log Z_{n,\tau}.
\]
By \eqref{eqn:entropic-smml-objective}, the latter equals
$\mathcal I_{n,\tau}(\mathcal P,q,\theta)$, proving the claim.
\end{proof}
Theorem~\ref{thm:variational} shows that entropic SMML can be interpreted as the optimization of the worst-case expected codelength over all data distributions $s$ that remain close to the prior predictive distribution $r_n$ in KL divergence. The parameter $\tau$ controls the strength of this robustness penalty: small $\tau$ enforces $s \approx r_n$, whereas large $\tau$ permits the tilted distribution to concentrate increasingly on high-codelength events. Closely related entropy-regularized minimax formulations also appear in entropic risk measurement and robust Bayesian decision theory, where worst-case expected loss is penalized by relative entropy with respect to a reference distribution; see Berger~\cite{Berger85} and F\"{o}llmer and Knispel~\cite{FollmerKnispel11}.
\subsection{PAC--Bayesian Interpretation}
\label{sec:pac:bayes}
The variational form from Theorem~\ref{thm:variational} immediately yields a PAC--Bayesian codelength bound. 
\begin{corollary}[PAC--Bayes-type codelength bound]
\label{cor:pac-bayes}
For any admissible codebook $(\mathcal P,q,\theta)\in\mathcal C_n$, any $\tau>0$, and any distribution $s\in\mathcal M(\mathcal X_n)$ with $s\ll r_n$,
\begin{align}
\label{eqn:pac-bayes-bound}
\mathbb{E}_s\!\left[\Lambda_{\mathcal P,q,\theta}(X_{n})\right]
\le
\mathcal I_{n,\tau}(\mathcal P,q,\theta)
+
\frac{1}{\tau} D_{\mathrm{KL}}(s\|r_n).
\end{align}
\end{corollary}
\begin{proof} This follows immediately from \eqref{eqn:variational-smml} by evaluating the supremum at the given distribution $s$.
\end{proof}
Corollary~\ref{cor:pac-bayes} identifies entropic SMML as a PAC--Bayes-type coding rule based on the standard KL change-of-measure inequality. Specifically, the entropic codelength is the smallest value that upper bounds expected codelength under every posterior tilt $s$, up to a KL complexity penalty relative to the prior predictive distribution. In this sense, entropic SMML is a KL-regularized robustification of ordinary SMML.
\subsection{Fixed-partition optimization}
The variational representation also clarifies the fixed-partition optimization. Let $\mathcal P=\{P_1,\dots,P_k\}$ be fixed, and define $C_{j,\tau}(\theta)$ as in \eqref{eqn:Ajtau}.
The next proposition gives the optimal assertion probabilities for fixed partition and codepoints.

\begin{proposition}[Optimal assertion probabilities for fixed partition]
\label{prop:q-opt}
Fix a partition $\mathcal P$ and codepoints $\theta=(\theta_1,\dots,\theta_k)$. Then any minimizer of \eqref{eqn:entropic-fixed-partition} over $q_1,\dots,q_k>0$ subject to $\sum_{j=1}^k q_j=1$ satisfies
\begin{align}
q_{j,\tau}^*
=
\frac{C_{j,\tau}(\theta_j)^{1/(1+\tau)}}
{\sum_{\ell=1}^k C_{\ell,\tau}(\theta_\ell)^{1/(1+\tau)}},
\qquad j=1,\dots,k.
\label{eqn:q-opt}
\end{align}
Substituting \eqref{eqn:q-opt} into \eqref{eqn:entropic-fixed-partition} yields the profiled entropic SMML criterion
\begin{align}
\mathcal I_{n,\tau}(\mathcal P,\theta)
=
\frac{1+\tau}{\tau}
\log
\left(
\sum_{j=1}^k
C_{j,\tau}(\theta_j)^{1/(1+\tau)}
\right).
\label{eqn:profiled-objective}
\end{align}
\end{proposition}
\begin{proof}
Since the logarithm is monotone, minimizing \eqref{eqn:entropic-fixed-partition} over $q$ is equivalent to minimizing
\[
F(q)
=
\sum_{j=1}^k q_j^{-\tau} C_{j,\tau}(\theta_j)
\]
subject to $\sum_j q_j=1$. The Lagrangian is
\[
\mathcal L(q,\lambda)
=
\sum_{j=1}^k q_j^{-\tau} C_{j,\tau}(\theta_j)
+
\lambda\Big(\sum_{j=1}^k q_j-1\Big),
\]
with first-order condition 
\[
-\tau C_{j,\tau}(\theta_j) q_j^{-\tau-1} + \lambda = 0,
\qquad j=1,\dots,k.
\]
This implies
\[
q_j
=
\left(\frac{\tau C_{j,\tau}(\theta_j)}{\lambda}\right)^{1/(1+\tau)}.
\]
Normalizing by $\sum_j q_j=1$ gives \eqref{eqn:q-opt}. Substituting the result back into $F(q)$ yields \eqref{eqn:profiled-objective}.
\end{proof}
Proposition~\ref{prop:q-opt} reduces the fixed-partition problem to the optimization of the profiled criterion \eqref{eqn:profiled-objective} over the codepoints alone. 
%
For fixed partition and codepoints with finite likelihood terms,  \eqref{eqn:q-opt} reduces to the ordinary SMML cell probabilities as \(\tau \downarrow 0\) (see \eqref{eqn:ordinary-q-opt}).
The corresponding endpoint limits are developed formally in the next section.
\section{SMML, Worst-Case Codelength, and Normalized Maximum Likelihood}
\label{sec:smml:nml}
This section establishes the two endpoint limits of the entropic SMML criterion. We show first that, as $\tau\downarrow 0$, entropic SMML reduces to the ordinary expected codelength criterion of strict minimum message length. We then show that, as $\tau\to\infty$, the objective converges to a worst-case codelength criterion. Since NML is characterized by minimax regret rather than by minimax absolute codelength, the precise connection to NML is obtained by applying the same soft-max limit to regret relative to the oracle maximum likelihood code. This is the classical setting of minimax-regret coding, as formalized by Shtarkov~\cite{Shtarkov87} and developed further by Rissanen~\cite{Rissanen07} and Gr\"{u}nwald~\cite{Grunwald07}.

Before proving the endpoint limits, we record a basic interpolation property of the entropic criterion. For each fixed codebook, the entropic codelength lies between the ordinary expected SMML codelength and the worst-case codelength, and it is monotone in the risk parameter $\tau$.
\begin{proposition}[Monotonicity and endpoint bounds]
Fix $n$ and an admissible codebook \((\mathcal P,q,\theta)\), and write
$
\Lambda(X_n)
=
\Lambda_{\mathcal P,q,\theta}(X_n)
$.
Let
\[
\mathcal D
=
\left\{
\tau>0:
\mathbb E_{r_n}\exp\{\tau\Lambda(X_n)\}<\infty
\right\}.
\]
For every \(\tau\in\mathcal D\),
\[
\mathbb E_{r_n}\Lambda(X_n)
\le
\mathcal I_{n,\tau}(\mathcal P,q,\theta)
\le
\operatorname*{ess\,sup}_{r_n}\Lambda(X_n),
\]
where the lower bound is interpreted whenever \(\mathbb E_{r_n}\Lambda(X_n)\) is well defined. Moreover, if
\(0<s<t\) and \(s,t\in\mathcal D\), then
\[
\mathcal I_{n,s}(\mathcal P,q,\theta)
\le
\mathcal I_{n,t}(\mathcal P,q,\theta).
\]
Thus \(\mathcal I_{n,\tau}(\mathcal P,q,\theta)\) is nondecreasing on its moment domain.
\end{proposition}

\begin{proof}
For brevity, write
$ \mathcal I_\tau = \mathcal I_{n,\tau}(\mathcal P,q,\theta) $. 
%
The lower bound follows from Jensen's inequality:
\[
\log \mathbb E_{r_n}\exp\{\tau\Lambda(X_n)\}
\ge
\mathbb E_{r_n}\log \exp\{\tau\Lambda(X_n)\}
=
\tau \mathbb E_{r_n}\Lambda(X_n).
\]
Dividing by $\tau$ gives $\mathcal I_\tau \ge \mathbb E_{r_n}\Lambda(X_n)$.
For the upper bound, let
$M=\operatorname*{ess\,sup}_{r_n}\Lambda(X_n)$.
If \(M=\infty\), the bound is trivial. Otherwise,
$\Lambda(X_n)\le M \; r_n\text{-a.s.}$ and hence
\[
\mathbb E_{r_n}\exp\{\tau\Lambda(X_n)\}
\le
\exp\{\tau M\}.
\]
Therefore $\mathcal I_\tau \le M$. It remains to prove monotonicity. Let \(0<s<t\), with \(s,t\in\mathcal D\), and set $Y=\exp\{\Lambda(X_n)\}$.  Then \(Y\ge0\), and
\[
\mathcal I_s
=
\log
\left(
\mathbb E_{r_n}Y^s
\right)^{1/s},
\qquad
\mathcal I_t
=
\log
\left(
\mathbb E_{r_n}Y^t
\right)^{1/t}.
\]
Since \(r_n\) is a probability measure, \(L^p\)-norms are nondecreasing in \(p\). Hence
\[
\left(
\mathbb E_{r_n}Y^s
\right)^{1/s}
\le
\left(
\mathbb E_{r_n}Y^t
\right)^{1/t}.
\]
Taking logarithms gives $\mathcal I_s\le \mathcal I_t$,  proving monotonicity.
\end{proof}
\subsection{Recovery of ordinary SMML}
For optimizer-convergence statements in this section, we use the following standard convention. Either the admissible class \(\mathcal C_n\) is finite, or it is equipped with a topology under which the relevant objectives are continuous, their sublevel sets are compact, and the displayed pointwise convergences are uniform over \(\mathcal C_n\). Under these conditions, convergence of minimizers follows from the usual argmin-continuity theorem. For purely pointwise statements about a fixed codebook, no compactness assumption on \(\mathcal C_n\) is required.
\begin{theorem}[Recovery of ordinary SMML]
\label{thm:tau-zero}
For any admissible codebook $(\mathcal P,q,\theta)\in\mathcal C_n$, such that \(\mathbb{E}_{r_n} \! \left[\exp\{u|\Lambda_{\mathcal P,q,\theta}(X_n)|\}\right] <\infty \)  for all \(u\) in a neighborhood of zero,
\begin{align}
\lim_{\tau \downarrow 0}
\mathcal I_{n,\tau}(\mathcal P,q,\theta)
=
\mathcal I_n(\mathcal P,q,\theta)
=
\mathbb E_{r_n}\!\left[
\Lambda_{\mathcal P,q,\theta}(X_{n})
\right].
\label{eqn:tau-zero-limit}
\end{align}
Moreover, under the uniform convergence and compactness convention stated above, if \((\mathcal P^*_{n,\tau},q^*_{n,\tau},\theta^*_{n,\tau})\) is any entropic SMML optimizer and the ordinary SMML optimizer is unique, then
\[
(\mathcal P_{n,\tau}^*,q_{n,\tau}^*,\theta_{n,\tau}^*)
\to
(\mathcal P_n^*,q_n^{*},\theta_n^{*})
\qquad \text{as }\tau\downarrow 0.
\]
\end{theorem}
\begin{proof}
Fix $(\mathcal P,q,\theta)\in\mathcal C_n$ and write
\[
\Lambda(X_{n}) := \Lambda_{\mathcal P,q,\theta}(X_{n}).
\]
By definition,
\[
\mathcal I_{n,\tau}(\mathcal P,q,\theta)
=
\frac{1}{\tau}\log
\mathbb E_{r_n}\!\left[e^{\tau \Lambda(X_{n})}\right].
\]
Expanding the logarithmic moment generating function around $\tau=0$ gives
\begin{align}
\mathbb E_{r_n}[\Lambda(X_{n})]
+
\frac{\tau}{2}\mathrm{Var}_{r_n}(\Lambda(X_{n}))
+
O(\tau^2).
\label{eqn:cumulant-expansion}
\end{align}
The leading term is the ordinary SMML objective $\mathcal I_n(\mathcal P,q,\theta)$ from \eqref{eqn:ordinary-smml-objective}, which proves \eqref{eqn:tau-zero-limit}. Under the uniform convergence/compactness convention stated above, the convergence of minimizers under uniqueness follows from the standard argmin-continuity theorem.
\end{proof}
Theorem~\ref{thm:tau-zero} shows that entropic SMML is a genuine risk-sensitive deformation of ordinary SMML rather than a different coding criterion. The first-order correction term in \eqref{eqn:cumulant-expansion} is proportional to the codelength variance under the prior predictive law, so entropic SMML may be viewed locally as a variance-penalized perturbation of ordinary SMML. In particular, the small-$\tau$ regime preserves the ordinary Bayesian coding interpretation of SMML, in which the expected two-part codelength is minimized under the prior predictive distribution.
\subsection{Worst-case codelength and regret-centered NML}
We now turn to the opposite endpoint $\tau\to\infty$. 
Let
\begin{align}
\label{eqn:Mn-def}
M_n(\mathcal P,q,\theta)
:=
\sup_{x\in\mathcal X_n}
\Lambda_{\mathcal P,q,\theta}(x)
\end{align}
denote the worst-case codelength associated with a codebook.
\begin{theorem}[Worst-case codelength limit]
\label{thm:tau-infty}
Fix \(n\) and an admissible codebook
\((\mathcal P,q,\theta)\in\mathcal C_n\). 
If \(M_n(\mathcal P,q,\theta)<\infty\), then
\[
\lim_{\tau\to\infty}
\mathcal I_{n,\tau}(\mathcal P,q,\theta)
=
M_n(\mathcal P,q,\theta).
\]
Moreover, if the above convergence is uniform over \(\mathcal C_n\), then
\[
\lim_{\tau\to\infty}
\inf_{(\mathcal P,q,\theta)\in\mathcal C_n}
\mathcal I_{n,\tau}(\mathcal P,q,\theta)
=
\inf_{(\mathcal P,q,\theta)\in\mathcal C_n}
M_n(\mathcal P,q,\theta).
\]
If the infima are attained, the same conclusion holds with
\(\inf\) replaced by \(\min\).
\end{theorem}
\begin{proof}
Fix \((\mathcal P,q,\theta)\) and write
\[
\Lambda(x)=\Lambda_{\mathcal P,q,\theta}(x),
\qquad
M_n=M_n(\mathcal P,q,\theta).
\]
Then
\[
\mathcal I_{n,\tau}
=
M_n
+
\frac{1}{\tau}
\log
\sum_{x\in\mathcal X_n}
r_n(x)e^{\tau(\Lambda(x)-M_n)}.
\]
Since \(\Lambda(x)-M_n\le 0\), the logarithmic term is nonpositive. Hence
\[
\mathcal I_{n,\tau}\le M_n.
\]
Conversely, for any \(\varepsilon>0\), choose
\(x_\varepsilon\in \mathcal X_n\) such that
\[
\Lambda(x_\varepsilon)\ge M_n-\varepsilon .
\]
Since \(r_n(x_\varepsilon)>0\),
\[
\mathcal I_{n,\tau}
\ge
M_n-\varepsilon
+
\frac{1}{\tau}\log r_n(x_\varepsilon).
\]
Taking \(\liminf_{\tau\to\infty}\) and then letting
\(\varepsilon\downarrow0\) gives
\[
\liminf_{\tau\to\infty}\mathcal I_{n,\tau}\ge M_n.
\]
Together with the upper bound, this proves the pointwise limit.

Uniform convergence over \(\mathcal C_n\) allows the limit and infimum to be interchanged, yielding the displayed identity. The statement with minima follows when the infima are attained.
\end{proof}
Theorem~\ref{thm:tau-infty} shows that entropic SMML converges pointwise to a minimax codelength objective. Specifically, the criterion chooses an SMML-style two-part codebook whose longest possible message is as short as possible. The remaining question is how this relates to the minimax \emph{regret} criterion that defines normalized maximum likelihood.

Let $\hat\theta(x)$ denote a maximum likelihood estimator and define the oracle ML codelength
\begin{align}
\Lambda_{\mathrm{ML}}(x)
=
-\log p_n(x\mid \hat\theta(x)).
\label{eq:ml-code}
\end{align}
For any codebook $(\mathcal P,q,\theta)$, define the regret relative to the ML code by
\begin{align}
R_{\mathcal P,q,\theta}(x)
&:=
\Lambda_{\mathcal P,q,\theta}(x)
-
\Lambda_{\mathrm{ML}}(x) \notag \\
&=
\log
\frac{
p_n(x\mid \hat\theta(x))
}{
q_{j(x)}p_n(x\mid \theta_{j(x)})
}.
\label{eq:regret-def}
\end{align}
With this notation,
\begin{align}
\sup_{x\in\mathcal X_n}
\Lambda_{\mathcal P,q,\theta}(x)
=
\sup_{x\in\mathcal X_n}
\bigl(
\Lambda_{\mathrm{ML}}(x)
+
R_{\mathcal P,q,\theta}(x)
\bigr).
\label{eq:wc-regret}
\end{align}
Equation~\eqref{eq:wc-regret} highlights the distinction between worst-case absolute codelength and worst-case regret. The following proposition makes this distinction precise in the large‑$\tau$ limit.
\begin{proposition}[Regret-centered endpoint and NML]
\label{prop:regret-nml}
Let \(\mu_n\) be a full-support probability distribution on \(\mathcal X_n\). For a coding distribution \(Q\) on \(\mathcal X_n\), define the regret
\[
R_Q(x)
=
-\log Q(x)
+
\log p_n(x\mid\hat\theta(x)),
\]
where \(\hat\theta(x)\) is a maximum likelihood estimate, and define
\[
\mathcal J_{n,\tau}(Q)
=
\frac{1}{\tau}
\log
\sum_{x\in\mathcal X_n}
\mu_n(x)\exp\{\tau R_Q(x)\}.
\]
If \(\sup_x R_Q(x)<\infty\), then
\[
\lim_{\tau\to\infty}
\mathcal J_{n,\tau}(Q)
=
\sup_{x\in\mathcal X_n}R_Q(x).
\]
Assume further that the Shtarkov sum
\[
S_n
=
\sum_{x\in\mathcal X_n}
p_n(x\mid\hat\theta(x))
\]
is finite. Then, when the admissible coding class is the set of all probability distributions on \(\mathcal X_n\), a minimax-regret solution is the normalized maximum likelihood distribution
\[
p_{\mathrm{NML}}(x)
=
\frac{p_n(x\mid\hat\theta(x))}{S_n},
\]
and the minimax regret is \(\log S_n\).
\end{proposition}
\begin{proof}
The first claim is the same soft-max argument as in
Theorem~\ref{thm:tau-infty}, applied to \(R_Q\) with reference measure
\(\mu_n\). Full support of \(\mu_n\) ensures that every near-maximizer of \(R_Q\) has positive \(\mu_n\)-mass.

For the minimax-regret claim, suppose that a coding distribution \(Q\)
satisfies
\[
\sup_x
\left\{
-\log Q(x)+\log p_n(x\mid\hat\theta(x))
\right\}
\le R .
\]
Then, for every \(x\),
\[
Q(x)
\ge
e^{-R}p_n(x\mid\hat\theta(x)).
\]
Summing over \(x\) gives
\[
1
=
\sum_x Q(x)
\ge
e^{-R}
\sum_x p_n(x\mid\hat\theta(x))
=
e^{-R}S_n.
\]
Thus \(R\ge\log S_n\). Equality is achieved by
\[
Q(x)=p_n(x\mid\hat\theta(x))/S_n.
\]
Therefore \(p_{\mathrm{NML}}\) is minimax optimal and the minimax regret is
\(\log S_n\).
\end{proof}
Proposition~\ref{prop:regret-nml} shows that applying the same entropic soft maximum to regret, rather than to absolute codelength, recovers the NML minimax-regret principle when the Shtarkov sum is finite.
When the Shtarkov sum is infinite, ordinary NML is not defined. In such cases, refined MDL often replaces NML by restricted, conditional, or luckiness-weighted variants; the latter introduce a nonnegative weight on the parameter space and lead to luckiness-normalized maximum likelihood (LNML) codes~\cite{Grunwald07,Miyaguchi2017,GrunwaldRoos19}. This suggests a natural extension of the regret-centered entropic construction in which the ML oracle term is augmented by a luckiness penalty.
\section{Joint $n$--$\tau$ Asymptotics}
\label{sec:n:tau:asymptotics}
Sections~\ref{sec:var:pac:bayes} and~\ref{sec:smml:nml} show that, for fixed sample size $n$, entropic SMML interpolates between ordinary SMML and a worst-case codelength criterion as $\tau$ ranges from $0$ to $\infty$, with NML recovered by the corresponding regret-centered construction. We now study the corresponding asymptotics when the sample size $n$ also tends to infinity. 
For each $n$, the entropic criterion admits an exact KL-regularized variational representation, so the main asymptotic question is how its behavior depends jointly on the growth of $\tau_n$ and on the scales governing typical codelength fluctuations and worst-case soft-max concentration under the prior predictive law. 
Throughout this section, let $\tau=\tau_n>0$ depend on $n$, and let
\[
\mathcal I_{n,\tau_n}(\mathcal P,q,\theta)
=
\frac{1}{\tau_n}
\log
\mathbb E_{r_n}
\!\left[
\exp\!\big(
\tau_n \Lambda_{\mathcal P,q,\theta}(X_{n})
\big)
\right]
\]
denote the entropic SMML objective. The theorem below is stated in terms of two auxiliary sequences, $v_n$ and $a_n$, which quantify these effects. Roughly speaking, $v_n$ controls the scale of stochastic fluctuations of $\Lambda_{\mathcal P,q,\theta}(X_n)$ about its mean under $r_n$, while $a_n$ controls the rate at which the entropic soft maximum concentrates on near worst-case observations. Precise definitions and assumptions are given in the appendix. 

%
\begin{theorem}[Joint $n$--$\tau$ asymptotics]
\label{thm:joint-ntau}
Let $\tau_n>0$ be a sequence, and suppose Assumptions~1 and~2 in the appendix hold. Then the following conclusions hold uniformly over admissible codebooks $(\mathcal P,q,\theta)\in\mathcal C_n$.

\begin{enumerate}
\item[(1)] \textbf{SMML regime.}
If $\tau_n\to 0$ and $\tau_n v_n \to 0$, then
\[
\mathcal  I_{n,\tau_n}(\mathcal P,q,\theta)
=
\mathbb E_{r_n}\!\left[\Lambda_{\mathcal P,q,\theta}(X_n)\right]
+o(1).
\]
In particular, any sequence of entropic SMML minimizers is asymptotically ordinary SMML optimal. 


\item[(2)] \textbf{Worst-case codelength regime.}
Under Assumption~2, if \(a_n/\tau_n\to0\), then
\[
\mathcal  I_{n,\tau_n}(\mathcal P,q,\theta)
=
\sup_{x\in\mathcal X_n} \Lambda_{\mathcal P,q,\theta}(x)
+o(1).
\]
%
\end{enumerate}
\end{theorem}

The proof is deferred to the appendix. Theorem~\ref{thm:joint-ntau} provides a scale-explicit characterization of the joint $n$--$\tau$ asymptotics. 
When \(\tau_n\to0\) and \(\tau_n v_n\to0\), entropic SMML is uniformly asymptotically equivalent to ordinary SMML over the admissible codebook class. When \(a_n/\tau_n\to0\), it is uniformly asymptotically equivalent to the
worst-case absolute codelength criterion on the prior-predictive support.
At the same time, for every $n$, the criterion admits an exact KL-regularized variational representation (see Theorem~\ref{thm:variational}), so the intermediate regime is naturally interpreted as a robust coding regime rather than as a separate asymptotic approximation. In settings where both the fluctuation scale and the soft-max penalty scale are logarithmic, that is, $v_n=O(\log n)$ and $a_n=O(\log n)$, the preceding bounds recover a logarithmic separation between the Bayesian average-case regime and the worst-case codelength regime.

The next section specializes the theory to regular exponential families, where the fixed-partition codepoint equation and the induced cell geometry simplify substantially.
\section{Regular Exponential Families}
\label{sec:exponential:families}
This section specializes entropic SMML to regular exponential families. In the ordinary SMML setting, codepoints in exponential families satisfy a moment-matching condition, and the exact fixed-codebook partition is the pullback of a polyhedral partition of sufficient-statistic space. These properties were established for ordinary SMML in~\cite{Dowty13,MakalicSchmidt26b}. We show here that the same affine cell structure survives under the entropic deformation, while the codepoint equation is replaced by a tilted moment-matching condition.

In the remainder of this section, unless stated otherwise, we work with a full regular exponential family in natural coordinates, with open convex natural parameter space \(\mathcal N\), writing
\begin{align}
p_n(x\mid\theta)
=
h(x)\exp\!\big(
\eta(\theta)^\top T(x)-n A(\eta(\theta))
\big),
\label{eqn:expfam}
\end{align}
where $T(x)\in\mathbb R^d$ is a sufficient statistic,
$\eta(\theta)\in\mathbb R^d$ is the natural parameter, and $A$ is the log-partition function. We write $\nu=\eta(\theta)$ when it is convenient to work directly in natural-parameter coordinates. For a fixed partition cell $P_j$, recall from \eqref{eqn:entropic-codepoint} that the entropic SMML codepoint solves
\begin{align}
\theta_{j,\tau}^*
\in
\argmin_{\theta\in \Theta}
\sum_{x\in P_j}
r_n(x)\,p_n(x\mid\theta)^{-\tau}.
\label{eqn:entropic-codepoint-expfam-start}
\end{align}
Substituting \eqref{eqn:expfam} into
\eqref{eqn:entropic-codepoint-expfam-start} gives
\begin{align}
p_n(x\mid\theta)^{-\tau}
=
h(x)^{-\tau}
\exp\!\big(
-\tau \eta(\theta)^\top T(x)
+
\tau n A(\eta(\theta))
\big),
\label{eq:expfam-tilt}
\end{align}
so the fixed-cell objective becomes
\begin{align}
C_{j,\tau}(\theta)
&=
\exp\!\big(\tau n A(\eta(\theta))\big)  \sum_{x\in P_j}
r_n(x)\,h(x)^{-\tau}
\exp\!\big(
-\tau \eta(\theta)^\top T(x)
\big).
\label{eq:Aj-expfam}
\end{align}
To express the first-order condition cleanly, let
\begin{align}
B_{j,\tau}(\nu)
:=
\sum_{x\in P_j}
r_n(x)\,h(x)^{-\tau}e^{-\tau \nu^\top T(x)},
\label{eq:Bj}
\end{align}
so that
\begin{align}
C_{j,\tau}(\nu)=e^{\tau n A(\nu)}B_{j,\tau}(\nu).
\label{eq:Ajnu}
\end{align}
Differentiating $\log C_{j,\tau}(\nu)$ with respect to $\nu$ yields
\begin{align}
\tau n \nabla A(\nu)
-
\tau
\frac{
\sum_{x\in P_j}
r_n(x)\,h(x)^{-\tau}e^{-\tau \nu^\top T(x)}T(x)
}{
\sum_{x\in P_j}
r_n(x)\,h(x)^{-\tau}e^{-\tau \nu^\top T(x)}
}.
\label{eq:logAj-derivative}
\end{align}
Setting \eqref{eq:logAj-derivative} equal to zero gives the following result.
\begin{proposition}[Tilted moment matching]
\label{prop:tilted-moment}
Let $\nu_{j,\tau}^*=\eta(\theta_{j,\tau}^*)$ be a minimizer of \(C_{j,\tau}(\nu)\). Then
\begin{align}
n\,\nabla A(\nu_{j,\tau}^*)
=
\sum_{x\in P_j}
w_{j,\tau}(x;\nu_{j,\tau}^*)\,T(x),
\label{eq:tilted-moment}
\end{align}
where
\begin{align}
w_{j,\tau}(x;\nu)
=
\frac{
r_n(x)\,h(x)^{-\tau}e^{-\tau \nu^\top T(x)}
}{
\sum_{y\in P_j}
r_n(y)\,h(y)^{-\tau}e^{-\tau \nu^\top T(y)}
},
\qquad x\in P_j.
\label{eq:tilted-weights}
\end{align}
In particular, in a regular canonical exponential family, the entropic SMML codepoint is the model parameter whose mean-value parameter matches a \emph{tilted} average of the sufficient statistic over the cell.
\end{proposition}
\begin{proof}
Equation~\eqref{eq:Ajnu} implies
\[
\log C_{j,\tau}(\nu)
=
\tau n A(\nu)+\log B_{j,\tau}(\nu).
\]
Differentiating both sides with respect to $\nu$ gives
\eqref{eq:logAj-derivative}. Setting the gradient equal to zero and normalizing the resulting weights yields \eqref{eq:tilted-moment}--\eqref{eq:tilted-weights}.
\end{proof}
Proposition~\ref{prop:tilted-moment} is the entropic analogue of the ordinary SMML moment-matching condition in canonical exponential families~\cite{MakalicSchmidt26b}. In the ordinary case, the codepoint matches the $r_n$-weighted average of the sufficient statistic over the cell; under the entropic deformation, the weights are exponentially tilted toward observations that are costly under the current codepoint. In the limit $\tau\downarrow 0$, the tilted weights reduce to the ordinary cellwise weights and \eqref{eq:tilted-moment} collapses to the classical moment-matching equation established in~\cite{MakalicSchmidt26b}.

Regular exponential families carry a dually flat information geometry in which KL divergence is the Bregman divergence generated by the log-partition function. In this geometry, ordinary SMML codepoints are KL/Bregman centroids of the cellwise distributions~\cite{MakalicSchmidt26b}. Under entropic SMML, the same geometric picture persists, but with the cellwise distribution replaced by the exponentially tilted weights~\eqref{eq:tilted-weights}. Thus the entropic codepoint is a \emph{tilted Bregman centroid}~\cite{BanerjeeEtAl05} that minimizes a risk-sensitive deformation of the ordinary cellwise KL projection. 
\begin{lemma}[Entropic SMML as a tilted \(m\)-projection]
\label{lem:em-projection}
Fix a partition cell \(P_j\), and let
\[
\tilde s_j(x)
=
\frac{r_n(x)}{r_n(P_j)},
\qquad x\in P_j ,
\]
and extend \(\tilde s_j\) by zero outside \(P_j\).
Then the ordinary SMML codepoint satisfies
\[
\theta_j^*
\in
\argmin_{\theta\in\Theta}
D_{\mathrm{KL}}\!\left(
\tilde s_j\,\middle\|\,p_n(\cdot\mid\theta)
\right).
\]
Now suppose \(p_n(\cdot\mid\nu)\) is a full regular exponential family with open convex natural parameter space \(\mathcal N\), and write
\[
p_n(x\mid\nu)
=
h(x)\exp\{\nu^\top T(x)-nA(\nu)\}.
\]
Let \(\nu_{j,\tau}^*\in\mathcal N\) be a minimizer of
\[
C_{j,\tau}(\nu)
=
\sum_{x\in P_j} r_n(x)p_n(x\mid\nu)^{-\tau}.
\]
Define
\[
s_{j,\tau}^*(x)
=
\frac{
r_n(x)p_n(x\mid\nu_{j,\tau}^*)^{-\tau}
}{
\sum_{y\in P_j} r_n(y)p_n(y\mid\nu_{j,\tau}^*)^{-\tau}
},
\qquad x\in P_j ,
\]
and extend \(s_{j,\tau}^*\) by zero outside \(P_j\). Then
\[
\nu_{j,\tau}^*
\in
\argmin_{\nu\in\mathcal N}
D_{\mathrm{KL}}\!\left(
s_{j,\tau}^*\,\middle\|\,p_n(\cdot\mid\nu)
\right).
\]
In particular, under a minimal representation of the full regular exponential family, this minimizer is unique.
\end{lemma}
\begin{proof}
For the ordinary SMML codepoint, minimizing the fixed-cell expected detail codelength over \(\theta\) is equivalent to minimizing
\[
\sum_{x\in P_j} r_n(x)[-\log p_n(x\mid\theta)].
\]
Since
\[
\sum_{x\in P_j} r_n(x)[-\log p_n(x\mid\theta)]
=
r_n(P_j)
\sum_{x\in P_j}
\tilde s_j(x)[-\log p_n(x\mid\theta)],
\]
and
\[
D_{\mathrm{KL}}\!\left(
\tilde s_j\,\middle\|\,p_n(\cdot\mid\theta)
\right)
=
\sum_{x\in P_j}
\tilde s_j(x)\log \tilde s_j(x)
-
\sum_{x\in P_j}
\tilde s_j(x)\log p_n(x\mid\theta),
\]
the two objectives differ only by a positive multiplicative constant and an additive constant independent of \(\theta\). Hence the ordinary SMML codepoint is the \(m\)-projection of \(\tilde s_j\) onto the model family.

Now consider the entropic criterion in natural coordinates. Let \(\nu_{j,\tau}^*\in\mathcal N\) be a minimizer of
\[
C_{j,\tau}(\nu)
=
\sum_{x\in P_j} r_n(x)p_n(x\mid\nu)^{-\tau}.
\]
Differentiating gives
\[
\nabla_\nu C_{j,\tau}(\nu)
=
-\tau
\sum_{x\in P_j}
r_n(x)p_n(x\mid\nu)^{-\tau}
\nabla_\nu \log p_n(x\mid\nu).
\]
At the minimizer \(\nu_{j,\tau}^*\in\mathcal N\),
\[
0
=
\sum_{x\in P_j}
s_{j,\tau}^*(x)
\nabla_\nu \log p_n(x\mid\nu_{j,\tau}^*).
\]
Since
\[
\nabla_\nu \log p_n(x\mid\nu)
=
T(x)-n\nabla A(\nu),
\]
this condition is equivalent to
\[
n\nabla A(\nu_{j,\tau}^*)
=
\sum_{x\in P_j}
s_{j,\tau}^*(x)T(x).
\]
On the other hand,
\[
D_{\mathrm{KL}}\!\left(
s_{j,\tau}^*\,\middle\|\,p_n(\cdot\mid\nu)
\right)
=
\mathrm{const}
-
\sum_{x\in P_j}
s_{j,\tau}^*(x)\log p_n(x\mid\nu).
\]
Using the exponential-family form,
\begin{align*}
-
\sum_{x\in P_j}
s_{j,\tau}^*(x)\log p_n(x\mid\nu)
&= \\
 nA(\nu)
-
&\nu^\top
\sum_{x\in P_j}s_{j,\tau}^*(x)T(x)
+
\mathrm{const}.
\end{align*}
Therefore
\[
\nabla_\nu
D_{\mathrm{KL}}\!\left(
s_{j,\tau}^*\,\middle\|\,p_n(\cdot\mid\nu)
\right)
=
n\nabla A(\nu)
-
\sum_{x\in P_j}s_{j,\tau}^*(x)T(x).
\]
This gradient vanishes at \(\nu=\nu_{j,\tau}^*\). 
The function
\[
\nu
\mapsto
D_{\mathrm{KL}}\!\left(
s_{j,\tau}^*\,\middle\|\,p_n(\cdot\mid\nu)
\right)
\]
is convex on the open convex natural parameter space \(\mathcal N\). Hence stationarity at \(\nu_{j,\tau}^*\) implies global optimality over \(\mathcal N\). Under a minimal representation, \(A\) is strictly convex, and hence the KL objective is strictly convex, so the minimizer is unique.
\end{proof}
\begin{remark}[Amari $e$--$m$ geometric interpretation]
Lemma~\ref{lem:em-projection} admits a direct interpretation in terms of Amari's dual affine connections~\cite{AmariNagaoka00}. Ordinary SMML codepoints arise as $m$-projections of the cellwise predictive distribution onto the model manifold, minimizing the forward Kullback--Leibler divergence. Entropic SMML applies the same $m$-projection after an exponential ($e$-geodesic) tilt of the predictive distribution induced by the entropic criterion. Thus, entropic SMML can be viewed as an $e$--$m$ procedure: an $e$-type exponential tilting of the source distribution followed by an $m$-projection onto the model manifold. As $\tau\downarrow 0$, the $e$-tilt vanishes, the tilted distribution converges to the original cellwise distribution, and the entropic SMML codepoints converge to the ordinary SMML $m$-projections. 
\end{remark}
\subsection{Affine Cell Structure}
We now consider the partition geometry for a fixed codebook. By~\eqref{eqn:fixed-codebook-partition}, an observation $x$ is assigned to cell $j$ rather than cell $\ell$ whenever
\begin{align}
-\log q_j - \log p_n(x\mid\theta_j)
\le
-\log q_\ell - \log p_n(x\mid\theta_\ell).
\label{eq:expfam-partition-start}
\end{align}
Substituting the exponential-family form \eqref{eqn:expfam} and canceling the common term $\log h(x)$ yields
\begin{align}
\big(\eta(\theta_j)-\eta(\theta_\ell)\big)^\top T(x)
\ge
\log\frac{q_\ell}{q_j}
+
n\Big(
A(\eta(\theta_j))-A(\eta(\theta_\ell))
\Big).
\label{eq:expfam-affine-boundary}
\end{align}
Thus each pairwise cell boundary is affine in the sufficient statistic $T(x)$.
%
\begin{theorem}[Polyhedral cells in sufficient-statistic space]
\label{thm:polyhedral-entropic}
Fix a codebook $(\mathcal P,q,\theta)\in\mathcal C_n$ for the exponential family \eqref{eqn:expfam}. For each $j=1,\dots,k$, define
\begin{align}
V_j
&=
\bigcap_{\ell\neq j}
\left\{
t\in\mathbb R^d:
\big(\eta(\theta_j)-\eta(\theta_\ell)\big)^\top t 
\right.
\left.
\ge
\log\frac{q_\ell}{q_j}
+
n\big(A(\eta(\theta_j))-A(\eta(\theta_\ell))\big)
\right\}.
\label{eq:Vj-polyhedral}
\end{align}
Then each $V_j$ is a convex polyhedron, and, up to ties on boundaries,
\begin{align}
P_j
=
\{x\in\mathcal X_n : T(x)\in V_j\}.
\label{eq:polyhedral-pullback}
\end{align}
\end{theorem}
\begin{proof}
For fixed $j$ and $\ell$, \eqref{eq:expfam-affine-boundary} defines a closed half-space in $t$-space. The intersection over $\ell\neq j$ is therefore a convex polyhedron. The pullback characterization~\eqref{eq:polyhedral-pullback} follows directly from the fixed-codebook decision rule.
\end{proof}
Theorem~\ref{thm:polyhedral-entropic} shows that the affine cell structure of ordinary exponential-family SMML~\cite{MakalicSchmidt26b} is unchanged under the entropic deformation. This result is consistent with Dowty's analysis of ordinary SMML for exponential families with continuous sufficient statistics~\cite{Dowty13}, in which the cells of an SMML estimator are characterized as convex polytopes; Theorem~\ref{thm:polyhedral-entropic} shows that the same affine-cell geometry persists under the entropic deformation considered here. The role of the risk parameter is not to alter the form of the partition rule, but to modify the location of the codepoints and the assertion weights through the tilted criterion. In this sense, the exponential-family geometry is stable across the entire entropic family. 
%
%
\begin{corollary}
If $\eta(\theta)=\theta$, then
\[
\nabla A(\theta_{j,\tau}^*)
=
\frac{1}{n}
\sum_{x\in P_j} w_{j,\tau}(x;\theta_{j,\tau}^*) T(x),
\]
so the mean-value parameter of the entropic codepoint is the tilted average of
the normalized sufficient statistic over the cell.
\end{corollary}
\begin{proof}
If $\eta(\theta)=\theta$, then $\nabla A(\theta)$ is the mean-value parameter of the regular canonical exponential family. Equation~\eqref{eq:tilted-moment} therefore yields the claim directly after dividing both sides by $n$.
\end{proof}
\begin{figure*}[t]
\centering
\includegraphics[width=0.95\textwidth]{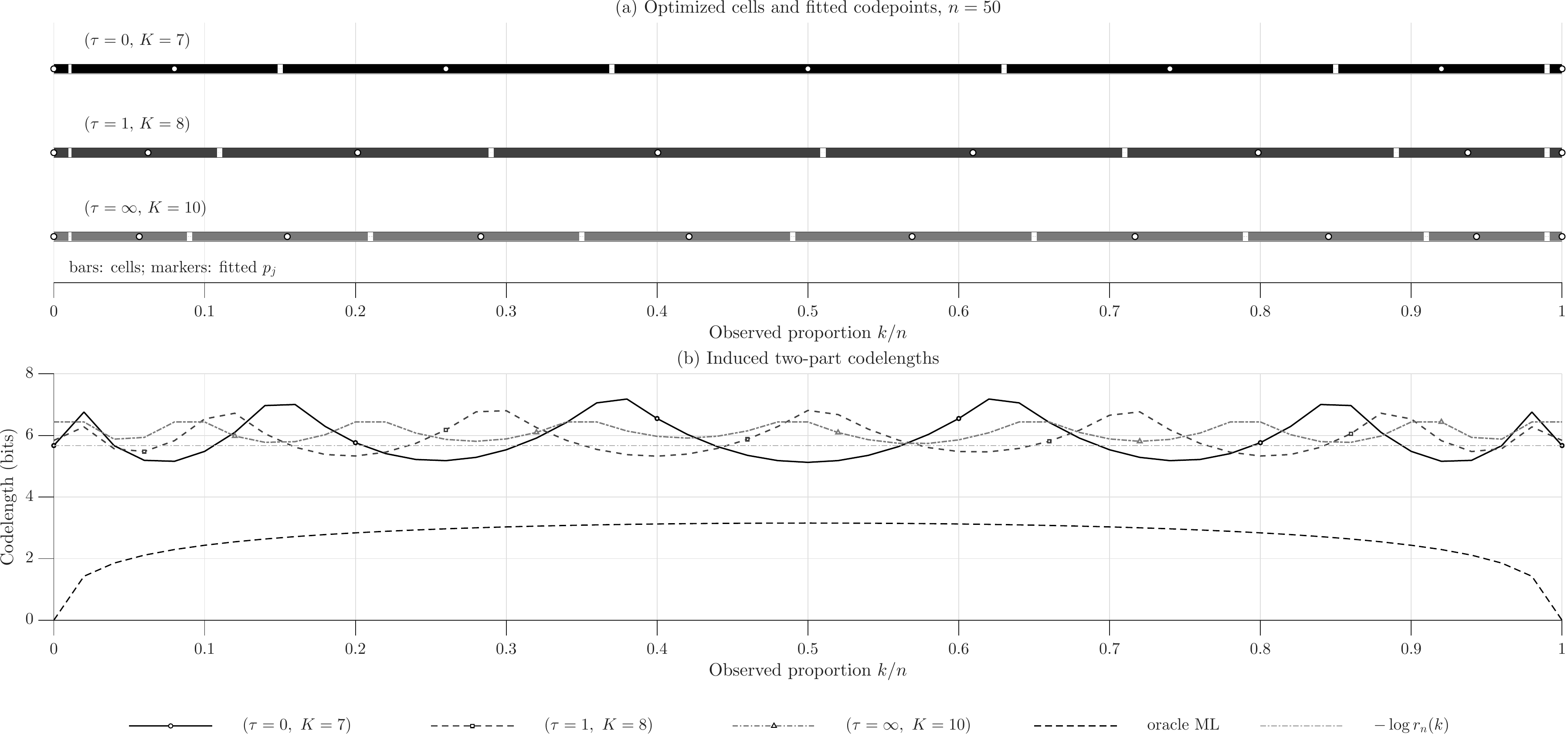}
\caption{Binomial comparison of ordinary SMML, entropic SMML, and the worst-case codelength endpoint for \(n=50\), with the number of cells optimized separately for each criterion. Panel (a) shows the optimized contiguous cells and fitted binomial codepoints. Panel (b) shows the induced two-part codelengths, together with the oracle maximum likelihood codelength and the prior-predictive ideal codelength \(-\log r_n(k)\). 
}
\label{fig:binomial-tau-comparison-freeK}
\end{figure*}
\subsection{A Binomial Example}
\label{sec:binomial-example}
As an illustration of the exponential family theory, consider the binomial model
\begin{align}
\label{eq:binom-model}
X_n \mid p \sim \mathrm{Binomial}(n,p),
\qquad
p\in(0,1),
\end{align}
with uniform prior
\begin{align}
\label{eq:binom-prior}
\pi(p)=1,\qquad 0<p<1.
\end{align}
The data space is $\mathcal X_n=\{0,1,\dots,n\}$, and the model likelihood is
\begin{align}
\label{eq:binom-likelihood}
p_n(x\mid p)
=
\binom{n}{x}p^x(1-p)^{n-x},
\qquad x=0,1,\dots,n.
\end{align}
This is a one-dimensional canonical exponential family with sufficient statistic $T(x)=x$, natural parameter
\[
\eta(p)=\log\frac{p}{1-p},
\]
and log-partition function
\[
A(\eta)=\log(1+e^\eta).
\]
The uniform prior yields an especially simple prior predictive distribution
\begin{align}
\label{eq:binom-prior-predictive}
r_n(x)=\frac{1}{n+1},
\qquad x=0,1,\dots,n.
\end{align}
Because the binomial model is a regular canonical exponential family, the general entropic SMML codepoint equation from Section~\ref{sec:exponential:families} simplifies to a scalar tilted moment equation.

\begin{proposition}[Binomial entropic codepoints and cell geometry]
\label{prop:binom-entropic}
Let $\mathcal P=\{P_1,\dots,P_k\}$ be a fixed partition of $\{0,1,\dots,n\}$, and let $p_j\in(0,1)$ denote the codepoint associated with cell $P_j$.

\begin{enumerate}
\item[\textnormal{1)}] The fixed-cell entropic objective is
\begin{align}
\label{eq:binom-cell-objective}
C_{j,\tau}(p)
=
\frac{1}{n+1}
\sum_{x\in P_j}
\binom{n}{x}^{-\tau}
p^{-\tau x}(1-p)^{-\tau(n-x)}.
\end{align}
\item[\textnormal{2)}] Any entropic SMML codepoint $p_{j,\tau}^*$ satisfies
\begin{align}
\label{eq:binom-fixed-point}
n\,p_{j,\tau}^*
=
\sum_{x\in P_j}
w_{j,\tau}(x;p_{j,\tau}^*)\,x,
\end{align}
where
\begin{align}
\label{eq:binom-weights}
w_{j,\tau}(x;p)
=
\frac{
\binom{n}{x}^{-\tau}
\left(\frac{p}{1-p}\right)^{-\tau x}
}{
\sum_{y\in P_j}
\binom{n}{y}^{-\tau}
\left(\frac{p}{1-p}\right)^{-\tau y}
},
\qquad x\in P_j.
\end{align}
\item[\textnormal{3)}] For a fixed codebook $(\mathcal P,q,p_1,\dots,p_k)$, each cell is an interval in $\{0,1,\dots,n\}$. More precisely, the pairwise boundary between cells $j$ and $\ell$ is determined by
\begin{align}
\label{eq:binom-boundary}
x\log\frac{p_j(1-p_\ell)}{p_\ell(1-p_j)}
\ge
\log\frac{q_\ell}{q_j}
+
n\log\frac{1-p_\ell}{1-p_j},
\end{align}
which is affine in $x$.
\end{enumerate}
\end{proposition}
\begin{proof}
Part~1 follows by substituting \eqref{eq:binom-likelihood} and \eqref{eq:binom-prior-predictive} into the definition of $C_{j,\tau}$. 
For Part~2, the canonical exponential family result~\eqref{eq:tilted-moment} gives
\[
n\,\nabla A(\eta(p_{j,\tau}^*))
=
\sum_{x\in P_j} w_{j,\tau}(x;\eta(p_{j,\tau}^*))\,x.
\]
Since $\nabla A(\eta(p))=p$ in the binomial model, this yields \eqref{eq:binom-fixed-point}, and the explicit weights are exactly \eqref{eq:binom-weights}.
For Part~3, the fixed-codebook decision rule compares
\[
-\log q_j-\log p_n(x\mid p_j)
\quad\text{and}\quad
-\log q_\ell-\log p_n(x\mid p_\ell).
\]
Substituting \eqref{eq:binom-likelihood} and canceling the common term $\log\binom{n}{x}$ gives \eqref{eq:binom-boundary}. Since the right-hand side is affine in $x$, each pairwise decision boundary is an interval boundary, and hence each cell is an interval in $\{0,1,\dots,n\}$.
\end{proof}

Proposition~\ref{prop:binom-entropic} makes the general theory especially transparent. The fixed-codebook partition remains one-dimensional and affine, so the polyhedral cell structure of Section~\ref{sec:exponential:families} reduces here to interval cells. At the same time, the entropic codepoint is the solution of the scalar fixed-point equation \eqref{eq:binom-fixed-point}, which is a tilted version of the ordinary SMML cellwise mean.

The ordinary SMML limit is recovered as $\tau\downarrow 0$. Indeed, the weights in \eqref{eq:binom-weights} become asymptotically uniform over each cell, so
\begin{align}
\label{eq:binom-smml-limit}
p_{j,\tau}^*
\to
\frac{1}{n|P_j|}
\sum_{x\in P_j} x,
\qquad
q_{j,\tau}^*
\to
\frac{|P_j|}{n+1},
\qquad \tau\downarrow 0,
\end{align}
recovering the ordinary SMML codebook for the binomial model. Thus, in this one-dimensional setting, entropic SMML can be viewed as a continuous deformation of the classical SMML interval partition toward increasingly robust codepoints and assertion weights.

Figure~\ref{fig:binomial-tau-comparison-freeK} illustrates the effect of the risk parameter in the binomial model for \(n=50\), with the number of cells optimized separately for each value of \(\tau\). As \(\tau\) increases from \(0\) to an intermediate value and then to \(\infty\), the optimized codebook shifts from minimizing expected codelength toward controlling the worst-case absolute codelength over the count space.
\section{Conclusion}
\label{sec:conclusion}
We introduced entropic SMML, a risk-sensitive generalization of strict minimum message length that replaces expected codelength under the prior predictive distribution with an exponential certainty equivalent. This yields a single criterion that interpolates smoothly between Bayesian average-case coding and worst-case codelength coding: ordinary SMML is recovered as $\tau\downarrow 0$, while the corresponding regret-centered limit recovers normalized maximum likelihood as $\tau\to\infty$.

The entropic formulation admits a KL-regularized variational representation and a PAC--Bayesian-type interpretation, positioning entropic SMML at the intersection of MML, PAC--Bayes, and MDL. We established a joint \(n\)--\(\tau\) asymptotic theory showing that, when the relevant fluctuation and soft-max scales are logarithmic, the average-case SMML regime holds for \(\tau_n\log n \to 0\), whereas the worst-case regime holds for \(\tau_n/\log n \to \infty\). For full regular exponential families, partitions remain affine in sufficient-statistic space, while codepoints satisfy a self-consistent tilted $m$-projection condition, equivalently a tilted Bregman-centroid condition.

Several directions remain open, including global geometric characterizations beyond fixed codebooks, broader divergence-based deformations, and non-asymptotic finite-blocklength bounds.
\section{Acknowledgments}
Generative AI tools (Claude Opus 4.6; Microsoft Copilot GPT-5.5) were used in the preparation of this manuscript for generation and exploration of ideas as well as language improvement.
\appendix
%
%
\section{Proof of Theorem~\ref{thm:joint-ntau}}
\label{sec:app}
This appendix states the assumptions and proof underlying Theorem~\ref{thm:joint-ntau} in a form that matches the scale-explicit formulation used in Section~\ref{sec:n:tau:asymptotics}. The assumptions isolate the two quantities that control the joint $n$--$\tau$ asymptotics: a fluctuation scale $v_n$ for centered codelengths and a soft-max penalty scale $a_n$ governing concentration on worst-case observations. 
\begin{assumption}[Local cumulant control]
There exist a sequence \(v_n>0\) and constants \(u_0>0\) and \(C_0<\infty\) such that, for all sufficiently large \(n\), every admissible codebook \(z=(\mathcal P,q,\theta)\in\mathcal C_n\) satisfies
\[
m_n(z)
:=
\mathbb E_{r_n}
\!\left[
\Lambda_{\mathcal P,q,\theta}(X_n)
\right]
<\infty,
\]
and, for every \(0\le u\le u_0\),
\[
\log \mathbb E_{r_n}\! \exp\!\Bigl(
u\bigl(\Lambda_{\mathcal P,q,\theta}(X_n)-m_n(z)\bigr)
\Bigr)
\le C_0 u^2 v_n .
\]
\end{assumption}

\begin{assumption}[Soft-max penalty control]
There exist sequences \(a_n>0\) and \(\varepsilon_n\downarrow0\) such that, for all sufficiently large \(n\) and every admissible codebook \(z=(\mathcal P,q,\theta)\in\mathcal C_n\), the quantity
\[
M_n(z)
:=
\sup_{x\in\mathcal X_n}
\Lambda_{\mathcal P,q,\theta}(x)
\]
is finite, and there exists a point \(x^\star_n=x^\star_n(z)\in\mathcal X_n\) satisfying
\[
\Lambda_{\mathcal P,q,\theta}(x^\star_n)
\ge
M_n(z)-\varepsilon_n,
\qquad
-\log r_n(x^\star_n)\le a_n .
\]
\end{assumption}
Assumptions 1-2 need only hold for asymptotically optimal codebooks; uniformity over the entire admissible class is adopted for simplicity.
\begin{proof}
Fix an admissible codebook $(\mathcal P,q,\theta)\in\mathcal C_n$, and write
\[
\Lambda_n := \Lambda_{\mathcal P,q,\theta}(X_n),
\qquad
m_n := \mathbb E_{r_n}[\Lambda_n].
\]

\textit{1) Small-risk regime.}
By definition,
\[
\mathcal I_{n,\tau_n}(\mathcal P,q,\theta)
=
\frac{1}{\tau_n}\log \mathbb E_{r_n} e^{\tau_n \Lambda_n}
=
m_n + \frac{1}{\tau_n}\log \mathbb E_{r_n} e^{\tau_n(\Lambda_n-m_n)}.
\]
If $\tau_n\to 0$ and $\tau_n v_n\to 0$, then for all sufficiently large $n$ we have $0\le \tau_n\le u_0$, and Assumption~1 yields
\[
\log \mathbb E_{r_n} e^{\tau_n(\Lambda_n-m_n)}
\le
C_0 \tau_n^2 v_n.
\]
Consequently,
\[
0
\le
\mathcal I_{n,\tau_n}(\mathcal P,q,\theta)-m_n
\le
C_0 \tau_n v_n
\to 0,
\]
which proves
\[
\mathcal I_{n,\tau_n}(\mathcal P,q,\theta)
=
\mathbb E_{r_n}[\Lambda_{\mathcal P,q,\theta}(X_n)]
+o(1)
\]
uniformly over admissible codebooks. The uniform \(o(1)\) bound implies asymptotic optimality of entropic minimizers for the ordinary SMML criterion. Indeed, if \(z_{n,\tau_n}^*\) minimizes \(\mathcal I_{n,\tau_n}\) and \(z_n^*\) minimizes \(\mathcal I_n\), then
\[
\mathcal I_n(z_{n,\tau_n}^*)
\le
\mathcal I_{n,\tau_n}(z_{n,\tau_n}^*)+o(1)
\le
\mathcal I_{n,\tau_n}(z_n^*)+o(1)
\le
\mathcal I_n(z_n^*)+o(1),
\]
and hence
\[
\mathcal I_n(z_{n,\tau_n}^*)-\inf_{z\in\mathcal C_n}\mathcal I_n(z)=o(1).
\]

\medskip

\textit{2) Large-risk regime.}
Let
\[
M_n(z)
:=
\sup_{x\in \mathcal X_n}
\Lambda_{\mathcal P,q,\theta}(x),
\qquad z=(\mathcal P,q,\theta).
\]
By Assumption~2, \(M_n(z)<\infty\), and there exists
\(x_n^\star\in\mathcal X_n\) satisfying
\[
\Lambda_{\mathcal P,q,\theta}(x_n^\star)
\ge M_n(z)-\varepsilon_n,
\qquad
-\log r_n(x_n^\star)\le a_n .
\]
Then
\begin{align*}
\mathcal I_{n,\tau_n}(\mathcal P,q,\theta)
&=
\frac{1}{\tau_n}
\log
\sum_{x\in \mathcal X_n}
r_n(x)e^{\tau_n \Lambda_{\mathcal P,q,\theta}(x)} \\
&=
M_n(\mathcal P,q,\theta) 
+
\frac{1}{\tau_n}
\log
\sum_{x\in \mathcal X_n}
r_n(x)e^{\tau_n(\Lambda_{\mathcal P,q,\theta}(x)-M_n(\mathcal P,q,\theta))}.
\end{align*}
Since $\Lambda_{\mathcal P,q,\theta}(x)-M_n(\mathcal P,q,\theta)\le 0$ for all $x$, the sum inside the logarithm is at most $1$, implying
\[
\mathcal I_{n,\tau_n}(\mathcal P,q,\theta)\le M_n(\mathcal P,q,\theta).
\]
On the other hand, retaining only the near-maximizing point \(x^\star_n\) from Assumption 2 yields
\[
\mathcal I_{n,\tau_n}(\mathcal P,q,\theta)
\ge
M_n(\mathcal P,q,\theta)-\varepsilon_n
+
\frac{1}{\tau_n}\log r_n(x^\star_n).
\]
By Assumption 2, \(-\log r_n(x^\star_n)\le a_n\), and therefore
\[
0
\le
M_n(\mathcal P,q,\theta)-\mathcal I_{n,\tau_n}(\mathcal P,q,\theta)
\le
\varepsilon_n+\frac{a_n}{\tau_n}.
\]
If \(a_n/ \tau_n \to 0\), it follows that
\[
\mathcal I_{n,\tau_n}(\mathcal P,q,\theta)
=
\sup_{x\in \mathcal X_n} \Lambda_{\mathcal P,q,\theta}(x)
+o(1)
\]
uniformly over admissible codebooks. This shows that the large-\(\tau_n\) regime reduces entropic SMML to a
worst-case codelength criterion on \(\mathcal X_n\). The corresponding regret-centered formulation recovers the NML principle under the standard Shtarkov finite-complexity condition and when the admissible coding class ranges over all probability assignments on \(\mathcal X_n\).
\end{proof}
%
%
%
%
\bibliographystyle{unsrt}
\bibliography{bibliography}  
\end{document}